\documentclass[reqno,a4paper,12pt]{amsart} 

\usepackage{amsmath,amscd,amsfonts,amssymb}
\usepackage{mathrsfs}

\usepackage{bbm}

\numberwithin{equation}{section}
\numberwithin{figure}{section}

\newcommand\R{\mathbb{R}}

\newcommand\Z{\mathbb{Z}}

\newcommand\T{\mathbb{T}}

\newcommand\1{\mathbbm{1}}
\newcommand\eps{\varepsilon}
\newcommand{\pphi}{\varphi}
\newcommand{\PP}{\mathbb{P}}

\renewcommand\le{\leqslant}

\renewcommand\leq{\leqslant}

\newcommand\sbt{\subset}

\newcommand{\ft}[1]{\widehat{#1}}
\newcommand{\dotprod}[2]{\langle #1 , #2 \rangle}

\theoremstyle{plain}
\newtheorem{thm}{Theorem}[section]
\newtheorem{lem}[thm]{Lemma}
\newtheorem{lemma}[thm]{Lemma}

\newtheorem*{claim*}{Claim}

\newcommand{\thmref}[1]{Theorem~\ref{#1}}

\newcommand{\lemref}[1]{Lemma~\ref{#1}}

\theoremstyle{definition}
\newtheorem{definition}[thm]{Definition}
\newtheorem*{definition*}{Definition}
\newtheorem*{remarks*}{Remarks}
\newtheorem*{remark*}{Remark}

\newenvironment{enumerate-alph}
{\begin{enumerate}
\addtolength{\itemsep}{5pt}
}
{\end{enumerate}}

\newenvironment{enumerate-num}
{\begin{enumerate}
\addtolength{\itemsep}{5pt}
}
{\end{enumerate}}

\newenvironment{enumerate-text}
{\begin{enumerate}
\addtolength{\itemsep}{5pt}
}
{\end{enumerate}}

\begin{document}

\title
[No unconditional Schauder frames of exponentials]
{There are no unconditional Schauder frames of exponentials in $L^p$ spaces, $p \ne 2$}

\author{Nir Lev}
\address{Department of Mathematics, Bar-Ilan University, Ramat-Gan 5290002, Israel}
\email{levnir@math.biu.ac.il}

\author{Anton Tselishchev}
\address{St. Petersburg Department of Steklov Mathematical Institute, Fontanka 27, St. Petersburg 191023, Russia}
\email{celis\_anton@pdmi.ras.ru}

\date{October 10, 2026}
\thanks{Research supported by ISF Grant No.\ 854/25. }

\begin{abstract}
We proved recently that there are no unconditional Schauder frames of exponentials in the space $L^p(\Omega)$, $1 \le p < 2$, for any set $\Omega \subset \mathbb{R}^d$ of finite measure. For $p > 2$, we proved that the same holds provided that $\Omega$ has nonempty interior, but the general case remained open. In this note we close this gap, i.e., we establish that exponentials cannot form an unconditional Schauder frame in $L^p(\Omega)$, $p \ne 2$, without any additional assumptions on the set $\Omega$ of finite measure other than measurability.
\end{abstract}

\maketitle


\section{Introduction}

\subsection{}
A system $\{u_n\}_{n=1}^{\infty}$ in a separable Banach space $X$ is called a \emph{Schauder basis} if every $x\in X$ admits a unique series expansion $x = \sum_{n=1}^{\infty} c_nu_n$. If the series converges unconditionally for every $x \in X$, then $\{u_n\}_{n=1}^{\infty}$ is called an \emph{unconditional basis}. 

The best known example of an unconditional basis in the space $L^p[0,1]$, $1 < p < +\infty$, is the Haar basis
(see \cite[II.B.9--10, II.D.13]{Woj91}). 
On the other hand, the classical exponential system with integer frequencies $\{ e^{2\pi i n x} \}$, $n \in \Z$,
\emph{does not} form an unconditional basis in the space $L^p[0,1]$,
$p \ne 2$, see \cite[Proposition II.D.9]{Woj91}.

Moreover, for any set $\Omega\subset\R^d$ of finite measure, 
there is no unconditional basis of exponentials
in the space $L^p(\Omega)$, $p \ne 2$.
This is a consequence of a more general result due to Gaposhkin  \cite{Gap58},
who proved that the space $L^p(\Omega)$,  $p \ne 2$, does not admit a normalized unconditional basis consisting of uniformly bounded functions.

\subsection{}

Let $X$ be a separable Banach space and $X^*$ be its dual space. A system of elements
$\{(u_n, u_n^*)\}_{n=1}^{\infty} \sbt X\times X^*$ is called a \emph{Schauder frame} (or a \emph{quasibasis}) if every $x\in X$ admits a series expansion
\begin{equation}
\label{eq:R1.1}
x=\sum_{n=1}^{\infty} u_n^*(x) u_n.
\end{equation}
If the series \eqref{eq:R1.1} 
converges unconditionally for every $x \in X$, then
$\{(u_n, u_n^*)\}_{n=1}^{\infty}$
is called an \emph{unconditional Schauder frame}.
We note that the functionals $\{u_n^*\}$ are not required to be biorthogonal to $\{u_n\}$, and in general the coefficients of the series expansion \eqref{eq:R1.1} may be not unique.

In our recent paper \cite{LT25b} we considered the question of existence of unconditional Schauder frames of exponentials in the spaces $L^p(\Omega)$, $p\neq 2$, where $\Omega\subset\R^d$ is a set of finite measure; we refer the reader to the introduction of \cite{LT25b} for a broader  motivation and context for this question.  By an exponential here we mean a function of the form 
\begin{equation}
e_\lambda(x) = e^{2\pi i \langle \lambda, x\rangle},
\quad x \in \R^d,
\end{equation}
where the frequency $\lambda$ is an arbitrary point in $\R^d$.

We proved in \cite{LT25b} that if $\Omega\subset\R^d$ is an arbitrary measurable set of positive and finite measure, then for $1 < p < 2$ the space $L^p(\Omega)$ does not admit any unconditional Schauder frame of exponentials. The same is true for $p=1$, see \cite[Section 4.3]{BC20}.

For $p>2$, the same result was proved in \cite{LT25b} under the extra assumption that $\Omega$ has nonempty interior, but the general case remained open. In this paper, we remove this additional assumption on the set $\Omega$ and prove the result in the general case.

\begin{thm}\label{thm:mainexp}
Let $\Omega\subset\R^d$ be a measurable set of positive and finite measure. Then the space $L^p(\Omega)$, $p>2$, does not admit any unconditional Schauder frame of exponentials, i.e.\ of the form 
$\{(e_\lambda, e_\lambda^*)\}$, $\lambda\in\Lambda$,
where $\Lambda\subset\R^d$ is any countable set
and $\{e_\lambda^*\}$ are arbitrary elements of the dual 
space $(L^p(\Omega))^*$.
\end{thm}

Thus the problem of existence of unconditional Schauder frames of exponentials in $L^p(\Omega)$ spaces is completely resolved in the negative for all $p \ne 2$.

To the contrary, we proved in \cite{LT25b} that for any $p>1$ there exists an unconditional Schauder frame
$\{(g_n, g_n^*)\}$  in  $L^p(\Omega)$ consisting of \emph{unimodular functions}, i.e.\ satisfying the condition
$|g_n(x)| = 1$ a.e.
In particular, this shows that Gaposhkin's result \cite{Gap58} fails if the word ``basis'' is replaced with ``Schauder frame''. We remark that the space $L^1$ does not admit any unconditional Schauder frames, see \cite[Section 4.3]{BC20}.

\subsection{}
Our motivation for studying this problem in part comes from the recent results on the corresponding problem for systems formed by translates of a single function. It was proved in \cite{OZ92}, \cite{OSSZ11}, \cite{FOSZ14} that a system of translates cannot form an unconditional basis in the space $L^p(\R)$ for any $p$. However, for $p>2$ there exists a system of translates which forms an unconditional Schauder frame in $L^p(\R)$ \cite{FOSZ14}. On the other hand, we proved in \cite{LT25a} that if $1\le p\le 2$ then no system of translates may form an unconditional Schauder frame for the space $L^p(\R)$. A similar problem for Gabor systems was addressed in our recent paper \cite{LT26}.

\subsection*{Disclosure of AI usage}
GPT-6 Astra found the initial proof of the main result.
The text of the paper was written by the authors. 
GitHub Copilot was used for inline suggestions and autocompletion during the writing and typesetting.


\section{Preliminaries}

In this section we review some basic facts that will be used later in the proof.

\subsection{}
We start with a basic criterion for unconditional convergence.

\begin{lemma}
\label{lem:unc_criterion}
Let $X$ be a reflexive Banach space, and $\{x_n\}_{n=1}^\infty \subset X$. Assume that there is a constant $K$ such that for any $N$ and arbitrary scalars $\{\theta_n\}$ with $|\theta_n|\le 1$, 
\begin{equation}
\label{eq:weakly_unconditionally_convergent}
\Big\| \sum_{n=1}^N \theta_n x_n \Big\| \leq K.
\end{equation}
Then the series $\sum_{n=1}^\infty   x_n$ converges unconditionally in $X$.
\end{lemma}

\begin{proof}
Indeed, the condition \eqref{eq:weakly_unconditionally_convergent} implies that the series  $\sum_{n=1}^\infty x_n$ is weakly unconditionally convergent (see \cite[II.D.3]{Woj91}). On the other hand, the space $X$ being reflexive does not contain any subspace isomorphic to $c_0$ (see \cite[Theorem II.A.14]{Woj91}). Hence every weakly unconditionally convergent series in $X$ is unconditionally convergent (see \cite[Proposition II.D.5]{Woj91}).
\end{proof}

\subsection{}
The next lemma states a basic property of unconditional Schauder frames, that can be proved
using the uniform boundedness principle in the same way 
as a similar statement for unconditional Schauder bases (see e.g.\ \cite[Proposition 3.1.3]{AK16}).

\begin{lemma}
\label{lem:unc_constant}
Let $\{(u_j, u_j^*)\}_{j=1}^\infty$ be an unconditional Schauder frame in 
a Banach space $X$.  Then there exists a constant $K$ such that for any 
$x\in X$ and for any
sequence of scalars $\{\theta_j\}$ satisfying  $|\theta_j|\le 1$, we have
\begin{equation}
	\label{eqtheta}
	\Big\| \sum_{j=1}^\infty \theta_j u_j^*(x) u_j \Big\| \leq K \|x\|
\end{equation}
and the series in \eqref{eqtheta} converges unconditionally.
\end{lemma}

In this case we say that  
$\{(u_j, u_j^*)\}_{j=1}^\infty$ 
is a \emph{$K$-unconditional} Schauder frame.

\subsection{}

A \emph{Rademacher sequence} is a sequence of i.i.d.\ random variables $\{\eps_n\}$ defined on some probability space $(\Omega, \PP)$ such that $\PP(\eps_n = 1) = \PP(\eps_n = -1) = 1/2$ for every $n$. We will use the following lemma, which formulates that fact that the space $L^p(\Omega)$, $1\le p\le 2$, has cotype $2$; see \cite[Theorem 6.2.14]{AK16} for a proof.

\begin{lemma}
\label{lem:cotype}
Let $1\le p \le 2$. There exists a positive constant $C_p$ such that for any finite 
sequence of functions $f_1, f_2, \ldots, f_N \in L^p(\Omega)$ we have 
\begin{equation}
\label{eq:cotype}
\Big( \sum_{j=1}^N \|f_j\|_p^2 \Big)^{1/2}\leq C_p \, \mathbb{E} \Big\|\sum_{j=1}^N \eps_j f_j\Big\|_p.
\end{equation}
\end{lemma}

\subsection{}

Following \cite{FOSZ14} we introduce the notion of an 
approximate Schauder frame.

\begin{definition}
Let $X$ be a separable Banach space with dual space $X^*$. 
A sequence of elements $\{(u_n, u_n^*)\} \subset X\times X^*$ is called an \emph{approximate Schauder frame} for $X$ if for any $x\in X$ the series
\begin{equation}\label{eq:approx_frame}
	Sx = \sum_n u_n^*(x) u_n
\end{equation}
converges in $X$ and defines a bounded and invertible linear operator $S: X \to X$. If additionally the series \eqref{eq:approx_frame} converges unconditionally for every $x\in X$, then the system 
$\{(u_n, u_n^*)\}$ is called an \emph{unconditional approximate Schauder frame}.
\end{definition} 

The following statement is easy to verify (see \cite[Lemma 3.1]{FOSZ14}).

\begin{lem}\label{lem:approxframe}
If $\{(u_n, u_n^*)\}$ is an 
unconditional approximate Schauder frame for $X$, 
then $\{(u_n, (S^{-1})^* u_n^*)\}$
forms an unconditional Schauder frame for $X$.
\end{lem}


\section{Stability under small perturbation of frequencies}

Let $\Omega\subset\R^d$ be a measurable set of positive and finite measure, and assume that there exist
two sequences $\{\lambda_n\}\subset \R^d$ and $\{g_n^*\} \subset L^p(\Omega)^*$ such that  the system $\{(e_{\lambda_n}, g_{n}^*)\}$ forms an unconditional Schauder frame in the space $L^p(\Omega)$, $p > 2$. Our goal is to show that this leads to a contradiction. 

By intersecting the set $\Omega$ with the cube $[-R,R]^d$ for a sufficiently large $R$, we may  assume in addition that $\Omega$ is a bounded set.

The next lemma shows that a small perturbation of the frequencies in the exponentials keeps the system an unconditional approximate Schauder frame.

\begin{lem}\label{lem:perturb}
Let $\Omega\subset\R^d$ be a bounded measurable set of positive measure, and assume that $\{(e_{\lambda_n}, g_{n}^*)\}$ forms an unconditional Schauder frame in $L^p(\Omega)$, $1 < p < + \infty$. There is $\eta > 0$ such that if a sequence $\{\mu_n\} \sbt \R^d$ satisfies $|\mu_n-\lambda_n| \le \eta$ for all $n$, then the system $\{(e_{\mu_n}, g_{n}^*)\}$ forms an unconditional approximate  Schauder frame in $L^p(\Omega)$.
\end{lem}

This result may be compared to the classical stability of exponential frames in the Hilbert space $L^2(\Omega)$, see e.g.\ \cite[Chapter 4, Section 8]{You01}. However, note that the classical notion of a frame in a Hilbert space is defined in terms of inequalities, and not merely in terms of an unconditionally convergent expansion of the form \eqref{eq:R1.1}.

\begin{proof}[Proof of Lemma~\ref{lem:perturb}]
We assume that $\Omega\subset [-R,R]^d$ and that the system $\{(e_{\lambda_n}, g_{n}^*)\}$ forms a $K$-unconditional Schauder frame in the space $L^p(\Omega)$. Let  $\{\mu_n\} \subset \R^d$ be an arbitrary sequence satisfying $|\mu_n-\lambda_n| \le \eta$ for all $n$.

We will use the Taylor series expansion of the exponential function in the form
\begin{equation}
	\label{eq:3.1}
	e_\lambda(x) = \sum_\alpha \frac{(2\pi i)^{|\alpha|}}{\alpha!} \lambda^\alpha  x^\alpha,
	\quad x   \in \R^d,
\end{equation}
where the sum goes over all multi-indices $\alpha$ with nonnegative integer coordinates. As usual, for a multi-index $\alpha=(\alpha_1, \alpha_2, \ldots \alpha_d)$ we denote
\begin{equation}
	x^\alpha = x_1^{\alpha_1} \ldots x_d^{\alpha_d}, \quad
	\alpha! = \alpha_1! \ldots \alpha_d!, \quad 
	|\alpha| = \alpha_1 + \ldots + \alpha_d.
\end{equation}

Let now $\theta_n$ be arbitrary scalars with $|\theta_n|\le 1$, and such that only finitely many $\theta_n$ are nonzero. Given $f\in L^p(\Omega)$, we have
\begin{align}
	& \sum_n \theta_n g_n^*(f) (e_{\mu_n}(x) - e_{\lambda_n}(x)) = \sum_n \theta_n g_n^*(f) e_{\lambda_n}(x)(e_{\mu_n-\lambda_n}(x) - 1) \label{eq:3.2.5}	 \\ 
		& \qquad = \sum_n \theta_n g_n^*(f) e_{\lambda_n}(x) \sum_{\alpha\neq 0} \frac{(2\pi i)^{|\alpha|}}{\alpha!}(\mu_n-\lambda_n)^{\alpha} x^\alpha. \label{eq:3.3}
\end{align}
Here we used the Taylor expansion \eqref{eq:3.1} in order to pass from \eqref{eq:3.2.5} to \eqref{eq:3.3}. The sum over $\alpha$ in \eqref{eq:3.3} converges uniformly on the bounded set $\Omega$.

Applying the triangle inequality and the estimate $|x^\alpha|\le R^{|\alpha|}$ for $x\in\Omega$, we obtain that the $L^p(\Omega)$ norm of the left hand side of \eqref{eq:3.2.5} does not exceed
\begin{align}
	&\sum_{\alpha\neq 0} \Big\|\frac{(2\pi i)^{|\alpha|}}{\alpha!} x^\alpha   \sum_n (\mu_n-\lambda_n)^\alpha \theta_n g_n^*(f)e_{\lambda_n}(x)  \Big\|_p\\
	& \qquad\le \sum_{\alpha\neq 0} \frac{(2\pi R)^{|\alpha|}}{\alpha!}   \Big\| \sum_n (\mu_n-\lambda_n)^\alpha \theta_n g_n^*(f)e_{\lambda_n} \Big\|_p.\label{eq:3.5}
\end{align}
Since $|(\mu_n-\lambda_n)^\alpha|\le \eta^{|\alpha|}$, the $K$-unconditionality of the system $\{(e_{\lambda_n}, g_n^*)\}$ implies that the quantity \eqref{eq:3.5} is not greater than
\begin{equation}
	\label{eq:3.6}
	K \|f\|_p \sum_{\alpha\neq 0} \frac{(2\pi \eta R)^{|\alpha|}}{\alpha!}   = K  (e^{2\pi d \eta R} - 1) \| f \|_p.
\end{equation}

Thus, given $\eps>0$, by choosing $\eta = \eta(K, R,d,\eps) > 0$ sufficiently small, we can make the right hand side of \eqref{eq:3.6} not greater than $\eps \|f\|_p$. We conclude that
\begin{equation}
	\label{eq:3.7}
	\Big\| \sum_n \theta_n g_n^*(f) (e_{\mu_n} - e_{\lambda_n})\Big\|_p\le \eps \|f\|_p 
\end{equation}
for any $f \in L^p(\Omega)$ and for arbitrary scalars $\theta_n$ with $|\theta_n| \le 1$, and such that only finitely many $\theta_n$ are nonzero.

We can now finish the proof by standard arguments. 
Indeed, by \eqref{eq:3.7} and
 \lemref{lem:unc_criterion},  the series
$\sum_n g_n^*(f) (e_{\mu_n} - e_{\lambda_n})$
converges unconditionally in $L^p(\Omega)$.
Clearly, the series $\sum_n  g_n^*(f) e_{\lambda_n}$ is also unconditionally convergent, 
since the system $\{(e_{\lambda_n}, g_{n}^*)\}$ forms an unconditional Schauder frame in $L^p(\Omega)$. It follows that the series
\begin{equation}
	Sf = 	\sum_n  g_n^*(f) e_{\mu_n}
\end{equation}
converges unconditionally in $L^p(\Omega)$, and defines a bounded linear operator  $S: L^p(\Omega) \to L^p(\Omega)$. Moreover, \eqref{eq:3.7} gives $\|S - I\| \le \eps$, where $I$ is the identity operator. Hence, by choosing $0 < \eps < 1$, we obtain that the operator $S$ is invertible. This means that  $\{(e_{\mu_n}, g_n^*)\}$ forms an unconditional approximate Schauder frame in $L^p(\Omega)$.
\end{proof}


\section{Transference to the torus}

The next step of the proof is to transfer the problem about unconditional Schauder frames of exponentials
to the $d$-dimensional torus $\T^d = \R^d / \Z^d$.

\begin{lem}
\label{lem:torus}
Assume that $\Omega\subset\R^d$ is a bounded, measurable set of positive measure, and that $\{(e_{\lambda_n}, g_{n}^*)\}$ forms an unconditional Schauder frame in $L^p(\Omega)$, $1 < p < + \infty$. Then there exists a measurable set $E \subset \T^d$ of positive measure, and a sequence $\{e^*_m\} \sbt L^{p}(E)^*$, such that $\{(e_m, e_m^*)\}$, $m\in\Z^d$, forms an unconditional Schauder frame in $L^p(E)$.
\end{lem}

\begin{proof}
Let $\eta$ be the number given by \lemref{lem:perturb}, and
choose a small $\gamma>0$. For each $n$, let $\mu_n$ be a closest point to $\lambda_n$ in the lattice $\gamma\Z^d$. If $\gamma $  is chosen sufficiently small, then we get
 $| \mu_n - \lambda_n| \le \eta$ for all $n$. Hence,
by \lemref{lem:perturb}, the system $\{(e_{\mu_n}, g_{n}^*)\}$ forms an unconditional approximate Schauder frame in $L^p(\Omega)$. In turn, an application of \lemref{lem:approxframe} gives us a sequence $\{h^*_n\} \subset L^p(\Omega)^*$ such that the system $\{(e_{\mu_n}, h_n^*)\}$ forms an unconditional Schauder frame in $L^p(\Omega)$.
	
It is possible that certain points of the lattice $\gamma\Z^d$ appear in the sequence $\{\mu_n\}$ several times. However by grouping the summands in the unconditionally convergent expansion
$f = \sum_n h_n^*(f) e_{\mu_n}$, and adding, if necessary, zeros as coefficient functionals, we get an unconditional Schauder frame of the form $\{(e_{\gamma m}, h_m^*)\}$, $m\in\Z^d$, for the space $L^p(\Omega)$.

Finally, after rescaling by $\gamma$, we obtain a new system $\{(e_m, e_m^*)\}$, $m\in\Z^d$, which forms an unconditional Schauder frame in $L^p(E)$, where $E := \gamma \Omega$. To be more precise, the new coefficient functionals $\{e_m^*\}$ are given by $e_m^* = U^* (h_m^*)$, where $U: L^p(E) \to L^p(\Omega)$ is the operator defined by $(Uf)(x) = f(\gamma x)$. 
It remains to notice that if $\gamma$ is sufficiently small then we get $ E \subset[-\frac12, \frac12]^d$, hence $E$ may be identified with a subset of $\T^d$.
\end{proof}


\section{An extension operator}

Thanks to \lemref{lem:torus}, it would suffice to prove 
that if $E\subset\T^d$ is an arbitrary measurable set of positive measure, and if $\{e_m^*\}$ are arbitrary elements of $L^p(E)^*$, then the system  $\{(e_m, e_m^*)\}$, $m\in\Z^d$, cannot form an unconditional Schauder frame in $L^p(E)$.

\subsection{}
Suppose to the contrary that 
 $\{(e_m, e_m^*)\}$, $m\in\Z^d$, is a
$K$-unconditional Schauder frame in $L^p(E)$, i.e.,
any $f\in L^p(E)$ admits an unconditionally convergent expansion 
\begin{equation}
\label{eq:unctorus}
f = \sum_{m\in\Z^d} e_m^*(f) e_m,
\end{equation}
and for arbitrary scalars $\theta_m$  with $|\theta_m|\le 1$, we have
\begin{equation}
\label{eq:unctorus_theta}
\Big\|\sum_{m\in\Z^d} \theta_m e_m^*(f) e_m\Big\|_{L^p(E)} \le K\|f\|_{L^p(E)}.
\end{equation}

The next step in the proof is to show that in this case, the series \eqref{eq:unctorus} in fact converges unconditionally in $L^p$  on the whole torus $\T^d$.

\begin{lem}\label{lem:extension}
For every $f\in L^p(E)$, the series
\begin{equation}
	\label{eq:Af}
	Af = \sum_{m\in\Z^d} e_m^*(f) e_m
\end{equation}
converges unconditionally in $L^p(\T^d)$. Moreover, 
\begin{equation}
	\label{eq:est_extension}
	\Big\|\sum_{m\in\Z^d} \theta_m e_m^*(f) e_m  \Big\|_{L^p(\T^d)}\le K m(E)^{-1/p}  \|f\|_{L^p(E)}
\end{equation}
for every choice of scalars $\theta_m$ with $|\theta_m|\le 1$.
\end{lem}

Since the series \eqref{eq:Af} converges to $f$ on the set $E$, the operator $A$ can therefore be viewed as a bounded extension operator, acting from $L^p(E)$ to $L^p(\T^d)$.

\begin{proof}[Proof of Lemma~\ref{lem:extension}]
Fix a function $f \in L^p(E)$, and let
$\theta_m$ be a sequence of scalars with $|\theta_m|\le 1$, 
and such that only finitely many $\theta_m$ are nonzero. Define
\begin{equation}
\pphi(x) = \sum_{m\in\Z^d} \theta_m e_m^*(f) e_m(x), \quad x \in \T^d,
\end{equation}
then $\pphi$ is a trigonometric polynomial. Using \eqref{eq:unctorus_theta}, for any $t \in \T^d$ we have
\begin{equation}
\int_{E} |\pphi(x-t)|^p dx 
= \Big\| \sum_{m\in\Z^d} \theta_m 
e^{-2\pi i \dotprod{m}{t}}
e_m^*(f) e_m\Big\|^p_{L^p(E)} \le K^p \|f\|^p_{L^p(E)}.
\end{equation}
Integrating this inequality with respect to the parameter $t \in \T^d$, we obtain
\begin{equation}
	m(E) \|\pphi\|_{L^p(\T^d)}^p = \int_E  \int_{\T^d}|\pphi(x-t)|^p dt dx = \int_{\T^d}\int_E |\pphi(x-t)|^p dx dt\le K^p \|f\|^p_{L^p(E)}.
\end{equation}
Thus, the estimate \eqref{eq:est_extension} holds
for any $f \in L^p(E)$ and any finitely supported sequence  $\{ \theta_m \}$ with $|\theta_m|\le 1$.
An application of \lemref{lem:unc_criterion} then implies the  unconditional convergence of the series \eqref{eq:Af}. In turn, this allows us to extend the estimate \eqref{eq:est_extension} from finitely supported sequences $\{\theta_m\}$ to all sequences with $|\theta_m|\le 1$.
\end{proof}

\subsection{}
\lemref{lem:extension} thus gives us  a bounded extension operator $A : L^p(E) \to L^p(\T^d)$ given by the unconditionally convergent series \eqref{eq:Af}. The dual operator may be viewed as an operator  $A^* : L^q(\T^d) \to L^q(E)$, where $q = p/(p-1)$.

\begin{lemma}
\label{lem:extension_dual}
The dual operator $A^* : L^q(\T^d) \to L^q(E)$ is given by
\begin{equation}
\label{eq:Astarg}
A^*g = \sum_{m\in\Z^d}  \ft{g}(-m) e_m^*, \quad g \in L^q(\T^d),
\end{equation}
and the series converges unconditionally in $L^q(E)$.  Moreover, 
\begin{equation}
	\label{eq:est_extension_dual}
	\Big\|\sum_{m\in\Z^d} \theta_m \ft{g}(-m) e_m^* \Big\|_{L^q(E)}\le K m(E)^{-1/p}  \|g\|_{L^q(\T^d)}
\end{equation}
for every choice of scalars $\theta_m$ with $|\theta_m|\le 1$.
\end{lemma}

Here, $\ft{g}(-m) =
\int_{\T^d} g(x) e_{m}(x) dx$ denote the Fourier coefficients of $g$.

\lemref{lem:extension_dual} can be proved by a standard duality argument and an application of \lemref{lem:unc_criterion}, and so we omit the details.

\subsection{}
Let us now put a Rademacher sequence instead of $\theta_m$ in \eqref{eq:est_extension_dual} and take expectation. Since $1<q < 2$, we can apply the cotype $2$ inequality (\lemref{lem:cotype}) in the space $L^q(E)$ and obtain the  estimate 
\begin{equation}
\label{eq:est_functionals_rad}
\Big(\sum_{m\in \Z^d} |\ft{g}(-m)|^2 \|e_m^*\|_{L^q(E)}^2\Big)^{1/2}
\le  K C_q  m(E)^{-1/p}  \|g\|_{L^q(\T^d)},
\end{equation}
where $g$ is an arbitrary function in $L^q(\T^d)$.


\section{A contradiction}

Let $I_N = \{-N,\ldots, N\}^d$ be the discrete cube of side length $2N+1$ in $\Z^d$. We use it to define a sequence $\{f_N\}$, $N=1,2,\ldots$, 
of functions in $L^p(\T^d)$, given by
\begin{equation}
\label{eq:f_eps_def}
f_N(x) = \frac1{|I_N|} \sum_{m \in I_N} 
e_{-m}(x) (Ae_m)(x), \quad x \in \T^d.
\end{equation}
We observe that  $\{f_N\}$ is a bounded sequence in the space $L^p(\T^d)$. Indeed, it follows from \eqref{eq:Af}, \eqref{eq:est_extension} that  $\|A e_m\|_{L^p(\T^d)} \le K$
for every $m \in \Z^d$, and therefore,
\begin{equation}
\label{eq:f_N_bdd}
\| f_N \|_{L^p(\T^d)} \le K, \quad N= 1,2,\ldots.
\end{equation}

Next, we claim that $f_N$ is  a weakly null sequence in $L^p(\T^d)$,  i.e., for every $g \in L^q(\T^d)$, $q=p/(p-1)$, we have $\int_{\T^d} f_N g \to 0$ as $N \to \infty$. 
Due to \eqref{eq:f_N_bdd}, it 
suffices to check this for each function $g$ 
belonging to some complete system of functions
in $L^q(\T^d)$. We will prove that this holds for any
exponential function $e_n$, $n \in\Z^d$.
In other words, it remains to show that 
for every $n \in \Z^d$ we have
$\ft{f}_N(-n)  \to 0$ as $N \to \infty$.

Indeed, given $n \in\Z^d$, we have
\begin{equation}
\ft{f}_N(-n) 
= \frac1{|I_N|} \sum_{m \in I_N}  \int_{\T^d} e_{n-m}(x) (Ae_m)(x)  dx 
= \frac1{|I_N|} \sum_{m\in I_N} e^*_{m-n}(e_m),
\end{equation}
where the first equality follows from \eqref{eq:f_eps_def},
while the second follows from \eqref{eq:Af}
and the orthogonality of the exponential functions $\{e_m\}$, $m\in\Z^d$, on the torus $\T^d$.

An application of the Cauchy--Schwarz inequality
(or Jensen's inequality) yields
\begin{equation}
\label{eq:CS_applied}
| \ft{f}_N(-n)  |^2 
\le  \frac1{|I_N|} 
\sum_{m\in I_N} |e^*_{m-n}(e_m)|^2 
\le  \frac{m(E)^{2/p}}{(2N+1)^d} 
\sum_{m\in I_N} \| e_{m-n}^* \|^2_{L^q(E)},
\end{equation}
where in the second inequality we used 
$\|e_m\|_{L^p(E)} = m(E)^{1/p}$ and $|I_N| = (2N+1)^d$.

Now, consider the $d$-dimensional Dirichlet kernel defined by
\begin{equation}
D_N(x) = \sum_{m\in I_N} e_m(x), \quad x \in \T^d.
\end{equation}
Then $\|D_N\|_{L^q(\T^d)}\le M_{d,q} N^{d(q-1)/q}$ where the constant $M_{d,q}$ depends only on $q$ and the dimension $d$.
Indeed, for $d=1$ this follows from the well-known pointwise estimate $|D_N(x)|\le C\min\{N, |x|^{-1}\}$ for $x\in [-\frac12, \frac12]$. In order to obtain the required estimate in all dimensions $d$, we simply multiply $d$ one-dimensional estimates.

Let us apply \eqref{eq:est_functionals_rad} with 
the function $g(x) = e_{n}(x) D_N(x)$. Since we have $\ft{g}(-m) = 1$ if $m+n\in I_N$, and $\ft{g}(-m) = 0$ otherwise, and using $p = q/(q-1)$, we obtain
\begin{equation}
\label{eq:est_functionals_sum}
 \sum_{m\in I_N} \| e_{m-n}^* \|^2_{L^q(E)} 
\le  K^2 C_q^2 M_{d,q}^2 m(E)^{-2/p}  N^{2d/p}.
\end{equation}
Using the assumption $p>2$, it follows
from \eqref{eq:CS_applied}, \eqref{eq:est_functionals_sum} that $\ft{f}_N(-n)  \to 0$ as $N \to \infty$.
This proves that $f_N$ is  a weakly null sequence in
the space $L^p(\T^d)$.

Finally, in order to obtain a contradiction, we
apply \eqref{eq:unctorus} to the function $f = e_m$
considered as an element of the space $L^p(E)$,
and conclude that $e_{-m}(x) (A e_m)(x) = 1$ a.e.\ on $E$. As a consequence, we get from \eqref{eq:f_eps_def} that $f_N(x) = 1$ a.e.\ on $E$. However, this contradicts the fact that $f_N$ is a weakly null sequence in $L^p(\T^d)$, e.g., by testing against the function $g = \1_E$. This contradiction completes the proof of \thmref{thm:mainexp}.
\qed


\end{document}